\documentclass{amsart}

\newcommand{\Alt}{\mathop{\mathrm{Alt}}}

\newcommand{\PSL}{\mathop{\mathrm{PSL}}}
\newcommand{\Sym}{\mathop{\mathrm{Sym}}}
\newcommand{\Irr}{\mathop{\mathrm{Irr}}}
\newcommand{\Aut}{\mathop{\mathrm{Aut}}}
\newtheorem{theo}{Theorem}[section]
\newtheorem{propo}[theo]{Proposition}
\newtheorem{lemma}[theo]{Lemma}
\newtheorem{coro}[theo]{Corollary}

\begin{document}
\title[Character degree sums  and Gelfand-Graev-like characters]{A uniform upper bound for the character degree sums  and Gelfand-Graev-like characters for finite simple groups}

\author[P. Spiga]{Pablo Spiga}

\author[A. Zalesski]{Alexandre Zalesski}
\address{Dipartimento di Matematica e Applicazioni,\newline
University of Milano-Bicocca,
 Via Cozzi 53 Milano, MI 20125, Italy}

\thanks{e-mail:
 E-mail: pablo.spiga@unimib.it (Pablo Spiga),
 alexandre.zalesski@gmail.com (Alexandre Zalesski)}


\maketitle

\begin{abstract} Let $G$ be a finite non-abelian simple group and let $p$ be a
prime. We classify all pairs  $(G,p)$ such that the sum of the
complex irreducible character degrees of $G$ is greater than  the index of a Sylow
$p$-subgroup of $G$. Our classification includes all groups of
Lie type in defining characteristic  $p$ (because every Gelfand-Graev character of  $G$ is multiplicity
free and has degree equal to the above index), and a handful of well-described examples.
\end{abstract} 

\medskip
\centerline{Dedicated to Daniela Nikolova on the occasion of her
60-th birthday}

\section{Introduction}
The problem of computing the sum $\Sigma(G)$ of the irreducible
character degrees of a finite group $G$ is  of considerable interest in representation
theory of finite groups. However, no explicit formula for  $\Sigma(G)$  is known for an arbitrary group $G$. The problem seems to be of greater importance for
simple groups. The irreducible character degrees of the sporadic simple groups and of the
exceptional groups of Lie type  have been computed, and hence with some
effort one can compute  $\Sigma(G)$. 
Nevertheless, for groups of large rank explicit formulae for $\Sigma(G)$ are not known, for instance $\Sigma(G)$ is not know for symplectic and orthogonal groups of even characteristic.
To the best of our knowledge, explicit
formulae  for  $\Sigma(G)$  have been obtained  for  $\mathrm{GL}_n(q)$~\cite{K,G1}, $\mathrm{PGL}_n(q)$~\cite{G1}, $\mathrm{Sp}_{2n}(q)$ with $q\equiv 1\pmod 4$~\cite{G2}, $\mathrm{Sp}_{2n}(q)$ with
$q\equiv -1\pmod 4$~\cite{V1}, $\mathrm{GU}_n(q)$~\cite{TV}, and for
orthogonal groups in odd characteristic~\cite{V}.  Observe that these results do not always yield a corresponding formula for the sum of the character degrees of the non-abelian simple composition factor of $G$.

For the symmetric
groups $\mathrm{Sym}(n)$, the sum of the degrees of the non-trivial irreducible
characters  equals the number of involutions, which can be
computed. However the analogous result fails for the alternating
groups.

\def\syl{Sylow $p$-subgroup }

Therefore in this paper we turn our attention to the problem of bounding from above $\Sigma(G)$, when $G$ is simple. Our upper bound is given in terms of the order of $G$ only, and is valid for the overwhelming majority of simple groups.
% a uniform  description of $\Sigma(G)$ in terms of the Sylow subgroups of $G$.

We let $|G|$ denote the order of a finite group $G$ and, given a prime $p$, we let $|G|_{p'}$ denote the index of a \syl of $G$. If $G$ is a group of Lie type in defining  characteristic $p$,
then $|G|_{p'}$ is a well-known lower bound for  $\Sigma(G)$. In fact, a Gelfand-Graev character $\Gamma$ of $G$ is multiplicity free and has degree   $|G|_{p'}$. Moreover, as the trivial character is not a constituent of $\Gamma$, we have $|G|_{p'}=\Gamma(1)<\Sigma(G)$.

In~\cite{Z}, the second author proposed  a generalization of the definition of
a Gelfand-Graev character. 
Namely, a character $\Gamma$ of an arbitrary group $G$ is said to be a Gelfand-Graev-like character for the prime $p$ if there exists a linear character $\nu$ of a Sylow $p$-subgroup of $G$ such that the induced character $\nu^G$ is multiplicity free and equals $\Gamma$. The investigation of this new notion  led to the following:

\smallskip

\noindent{\bf Conjecture.}  {\sl Let $G$ be a finite non-abelian simple group, let $p$ be a prime and let $P$ be a Sylow $p$-subgroup of $G$. Then
either 
\begin{description}
\item[$(1)$] $\Sigma(G)<|G|_{p'}=|G|/|P|, $ or
\item[$(2)$] $G$ has a Gelfand-Graev-like character for the prime $p$.
%there exists a linear character $\nu$ of $P$ such that
%the induced character $\nu^G$ is multiplicity free.
\end{description}}

\smallskip

In this paper we prove this conjecture, and in fact we obtain the following
upper bound for the sum of the irreducible character degrees, which will easily imply the conjecture:

%Our main result can be stated as follows:

\begin{theo}\label{ta1} Let G be a finite non-abelian simple group, let 
 $p$ be a prime and let  $P$ be a \syl of $G$.
 Then either $\Sigma(G)<|G|_{p'}=|G|/|P|$ or one of the following holds:
\begin{description}
\item[$(1)$] G is  isomorphic to a group of Lie type
 in defining characteristic $p;$
\item[$(2)$] $G\cong \mathrm{PSL}_2(q)$, $q>2$ even and $|P|=q+1;$
\item[$(3)$] $G\cong \mathrm{PSL}_2(q)$, $q>3$ odd, $p=2$, and $|P|=q- 1$ or $|P|=q+1$.
\end{description}
 \end{theo}

Note that in the exceptional cases (1), (2), (3) of Theorem~\ref{ta1} we have
$\Sigma(G)>|G|_{p'}$. This was observed in~\cite{Z}, where the
following  was proven:

\begin{theo}\label{ta2} Let $G,p$ and $P$ be as in Theorem~$\ref{ta1}$.
Then the following statements are equivalent:
\begin{description}
\item[$(A)$] the pair $(G,p)$ is as in $(1),(2)$ or $(3)$ of Theorem~$\ref{ta1}$;
\item[$(B)$]  $G$ has a Gelfand-Graev-like character for the prime $p$.
%there exists a linear character $\nu$ of $P$ such that the induced character $\nu^G$ is multiplicity free.
\end{description}
\end{theo}

The above conjecture follows from  Theorems~\ref{ta1} and~\ref{ta2}. %In fact, Theorem~\ref{ta2} follows from Theorem
%\ref{ta1} as the degree of $\nu^G$ equals $|G|_{p'}$. 
The upper
bound $|G|_{p'}$ for $\Sigma(G)$ is uniform in the sense that it is valid for
almost every pair  $(G,p)$. Of course, each simple group of Lie
type in characteristic $p$ stands aside because classical
Gelfand-Graev characters are multiplicity free and are of the shape
$\nu^G$,  for suitable $\nu$. Thus, the inequality $|G|_{p'}>
\Sigma(G)$ characterizes the simple groups of Lie type in defining
characteristic $p$, except for the cases (2) and~(3) of Theorem~\ref{ta1}.   

Using well-known isomorphisms between simple
groups of different series, we can restate Theorem~\ref{ta1} as
follows:
\begin{theo}\label{main1}
Let $G$ be a finite non-abelian simple group and let $p$ be a prime. If $G$ is an alternating or a sporadic simple group, then either $\Sigma(G)<|G|_{p'}$ or
\begin{description}
\item[$(1)$]$G=\Alt(5)\cong \mathrm{PSL}_2(4)\cong\mathrm{PSL}_2(5)$ and $p\in \{2,5\}$, or
\item[$(2)$]$G=\Alt(6)\cong \mathrm{PSL}_2(9)\cong(\mathrm{Sp}_4(2))'$ and $p\in \{2,3\}$, or
\item[$(3)$]$G=\Alt(8)\cong\mathrm{PSL}_4(2)$ and $p=2$.
\end{description}
If $G$ is a group of Lie type in characteristic $\ell$ with $\ell\neq p$, then either $\Sigma(G)<|G|_{p'}$ or
\begin{description}
\item[$(4)$]$G=(G_2(2))'\cong \mathrm{PSU}_3(3)$ and $p=3$, or
\item[$(5)$]$G=({^2}G_2(3))'\cong \mathrm{PSL}_2(8)$ and $p=2$, or
\item[$(6)$]$G=\mathrm{PSU}_3(3)\cong (G_2(2))'$ and $p=2$, or
\item[$(7)$]$G=\mathrm{PSU}_4(2)\cong \mathrm{PSp}_4(3)$ and $p=3$, or
\item[$(8)$]$G=(\mathrm{PSp}_4(2))'\cong\mathrm{PSL}_2(9)$ and $p=3$, or
\item[$(9)$]$G=\mathrm{PSp}_4(3)\cong \mathrm{PSU}_4(2)$ and $p=2$, or
\item[$(10)$]$G=\mathrm{PSL}_3(2)\cong \mathrm{PSL}_2(7)$ and $p=7$, or
\item[$(11)$]$G=\mathrm{PSL}_2(q)$, $\ell=2$, $q+1=p^t$ for some $t\geq 1$, or
\item[$(12)$]$G=\mathrm{PSL}_2(q)$, $\ell>2$, $p=2$ and $q-1=2^t$ for some $t\geq 2$, or
\item[$(13)$]$G=\mathrm{PSL}_2(q)$, $\ell>2$, $p=2$ and $q+1=2^t$ for some $t\geq 3$.
\end{description}
\end{theo}
When $G$ is a simple group of Lie type in characteristic $p$, we have $\Sigma(G)>|G|_{p'}$, and hence one wishes to bound $\Sigma(G)-|G|_{p'}$ or $\Sigma(G)/|G|_{p'}$. In fact, the main ingredients of our proof of Theorem~\ref{ta1} are Propositions~\ref{exp1} and~\ref{np1}. Namely, when $G$ is an exceptional group of Lie type Proposition~\ref{exp1} shows that $\Sigma(G)\leq 2\cdot|G|_{p'}$. Similarly, when $G$ is a classical group,
 Proposition~\ref{np1} shows that $\Sigma(G)\leq |W|\cdot|G|_{p'}$,  where $W$ is the Weyl group of the corresponding algebraic group $\mathbf{G}$. To prove  Proposition~\ref{np1} we show that the number of elements in any Lusztig series of irreducible characters of $G$ does not exceeds $|W|$, see Theorem~\ref{e(s)}.

In the literature there are already some upper bounds for
$\Sigma(G)$ in terms of the field parameter $q$ of $G$. In fact, for
classical groups of odd characteristic, it is shown that $\Sigma(G)\leq
(q+1)^{(r+d)/2}$, where $r$ and $d$ are the rank % and $d$
and the dimension of $\mathbf{G}$, see  Kowalski~\cite{Ko} and Vinroot \cite[Theorem 6.1]{V}. The same  bound holds  also for general linear and unitary groups in any
characteristic. For non-twisted groups $G$, Kowalski~\cite{Ko} has also shown the weaker estimate $\Sigma(G)\leq (q+1)^{(r+d)/2}\left(1+2r|W|/(q-1)\right)$.

\smallskip

\noindent{\bf Notation and conventions. }
Given a positive integer $n$, we
denote by $\Alt(n)$ and by $\Sym(n)$ the alternating group and the
symmetric group of degree $n$, respectively. For the other simple groups we simply follow the notation in~\cite{Atl}.
 
Given a
prime $p$, we denote by $n_p$ the largest power of $p$ dividing
$n$ and by $n_{p'}$ the $p'$-part of $n$, that is, $n_{p'}=n/n_p$.

For a finite group  $G$, we let $|G|$ denote the order of $G$, also  we let $G'$ denote the derived subgroup of $G$.
As usual, $\Irr (G)$ denotes the set of the complex  irreducible characters of  $G$. We write $\Sigma(G)$ for $\sum_{\chi\in \Irr(G)}\chi(1)$.
If $H$ is a subgroup of $G$ and $\nu$ is a character of $H$, then
$\nu^G$ means the induced character from $H$ to $G$.

Let ${\mathbf H}$ be a reductive algebraic group. An algebraic
group endomorphism $F:{\mathbf H}\rightarrow{\mathbf H} $ is
called Frobenius if the subgroup  $H={\mathbf H}^{F}$ of the
elements fixed by  $F$ is finite. The characteristic $\ell$, say,
of the ground field of ${\mathbf H}$ is called the defining
characteristic of ${\mathbf H}$, as well as of $H$. If ${\mathbf
H}$ is simple, we call $H$ a group of Lie type, and (when $H$ is not soluble) the unique non-abelian
simple composition factor $S$ of  $H$ is called a simple group of Lie
type. We refer to $\ell$ as the defining characteristic of $S$.

An abstract non-abelian simple group $G$ is said to be a finite simple group of Lie type if $G
$ is isomorphic to the simple composition factor of some group $H$ obtained
from a simple algebraic group ${\mathbf H}$. In particular, some group $G$
may have more than one defining characteristic. For instance, the alternating group $\Alt(6)$ is
isomorphic to $(\mathrm{Sp}_4(2))'$ and to $\mathrm{PSL}_2(9)$, and hence it is a
group of Lie type in defining characteristic $2$ and $3$. More examples can be deduced from Theorem~\ref{main1}.

\section{Preliminaries}\label{blabbering}
 Some results of general nature on the character degree sum can be obtained from the theory
of Frobenius-Schur indicator. We outline here this method.

Let $G$ be a finite group and let $M$ be an  irreducible  $\mathbb{C}G$-module affording  a character $\phi$. According to~\cite[\S 73A]{CR2}, if $M$ does not admit a non-zero $G$-invariant bilinear form, then  $M$ and $\phi$ are called unitary. This case occurs exactly when $\phi$ is not real-valued, see~\cite[page~$58$]{Is}. Similarly, if $M$ admits a non-zero $G$-invariant bilinear form $B$, then $M$ and $\phi$ are said to be orthogonal
 (respectively, symplectic) if $B$ is symmetric (respectively, skew-symmetric).

We let $\Irr^+ (G)$ and $\Irr^-( G)$ denote the set of orthogonal and  symplectic complex irreducible characters of $G$. The following result can be found
in~\cite[(4.6)]{Is}.

\begin{lemma}\label{is1} Let $t$ be the number of solutions of the equation $x^2=1$ in $G$.
Then $t=\sum_{\chi\in\Irr ^+(G)}\chi(1)-\sum_{\chi\in\Irr ^-(G)}\chi(1)$.
In particular, if $\Irr^- (G)=\emptyset$, then $t=\sum_{\chi\in\Irr^+ (G)}\chi(1)$.
Moreover, if every irreducible representation of $G$ is orthogonal, then $t=\sum_{\chi\in\Irr (G)}\chi(1)=\Sigma(G)$.
\end{lemma}

It is well-known that an irreducible representation
is orthogonal if and only if it can be realized over the real number field, see~\cite[(4.15)]{Is}. Furthermore,
$\Irr( G)=\Irr^+ (G)\cup \Irr^-( G)$ if and only if all characters of $G$ are real valued, see~\cite[page~$58$]{Is};
in turn, this happens if and only if every element of $G$ is conjugate to its own inverse~\cite[(6.13)]{Is}.

We include in this preliminary section  a result on the sporadic simple groups.

\begin{lemma}\label{is2} Let $S$ be a sporadic simple group and let $G=\Aut(S)$. Then, for every prime $p$, we have that either $\Sigma(G)<|S|_{p'}$, or $S=M_{12}$ and $p=2$. Furthermore, $\Sigma(M_{12})<|M_{12}|_{2'}$.
\end{lemma}
\begin{proof}
This follows from an immediate  inspection of the character table of $G$ in~\cite{Atl}.
\end{proof}

Let $H$ be $\mathbf{H}^{F}$, where $\mathbf{H}$ is a   connected reductive algebraic group in defining characteristic $\ell$ and $F:{\mathbf H}\rightarrow {\mathbf H}$ is a Frobenius endomorphism.
If ${\mathbf T}$ is an $F$-stable maximal torus of ${\mathbf H}$,
then $T=H\cap {\mathbf T}$ is called a maximal torus of $H$.
Recall that $W=N_{{\mathbf H}}({\mathbf T})/{\mathbf T}$ is a
finite group called the Weyl group of ${\mathbf H}$.

\begin{lemma}\label{ss1}
 Let $p$ be a prime with $p\neq \ell$ and let $P$ be a \syl of $H$. Then there exists
an $F$-stable maximal torus ${\mathbf T}$ of ${\mathbf H}$ with $P\leq N_{{\mathbf H}}({\mathbf T})\cap H$. Furthermore,
$N_{{\mathbf H}}({\mathbf T})\cap H=(N_{{\mathbf H}}({\mathbf
T}))^{F}$ and $(N_{{\mathbf H}}({\mathbf T})/{\mathbf
T})^{F}=W^{F}$. In particular, $|P|\leq |T|_p\cdot |W|_p\leq
|T|\cdot |W|$ for a suitable maximal torus $T$ of $H$.
\end{lemma}

\begin{proof}The first assertion is  in \cite[II-E.5.19]{SS},
and the second is a special case of \cite[3.13]{DM}.
\end{proof}

\begin{lemma}\label{tt2}
 If  ${\mathbf H}$ has rank $r$ and if $T$ is a maximal torus of $H$, then $|T|\leq (q_0+1)^r$, where $q_0$ is the 
absolute value of an eigenvalue of $F$ in its action on the
weight lattice of ${\mathbf H}$.
\end{lemma}

\begin{proof}Note that $q_0$ is defined in~\cite[page~35]{Ca}. If $H$ is
non-twisted, then the proof can be found in~\cite[page~75]{Ko}. The
argument for the general case is similar. Indeed, we have $F=q_0\cdot
\sigma$, where $\sigma$ is an automorphism of finite order of the
${\mathbb Z}$-lattice $Y_0$ defined in
\cite[page 85]{Ca}. Recall that there exists a surjective correspondence between the elements of the Weyl group $W$ and the $H$-conjugacy classes of maximal tori of $H$, see \cite[(3.3.3)]{Ca}. Therefore, every maximal torus (up to conjugacy) is determined by some $w\in W$.

By \cite[(3.3.5)]{Ca} and its proof, $|T|=|\det
(q_0\cdot {\rm Id} - \sigma^{-1} w)|$, where $w\in W$ is an element
defining $T$. It is known that $\sigma$ normalizes $W$, see, for instance,~\cite[comments leading to Theorem~32]{St}. Thus, $\sigma^{-1} w$ is of finite order, and hence diagonalizable over ${\mathbb C}$. Moreover, the 
eigenvalues of $\sigma^{-1}w$ are roots of unity. Therefore,
$|T|=\prod _{i=1}^r (q_0-\varepsilon_i)$, where $\varepsilon_1,\ldots
,\varepsilon_r$ are  the eigenvalues of $\sigma^{-1} w$, see~\cite[(3.3.5)]{Ca}. As $|q_0-\varepsilon_i|\leq
 q_0+1$, the lemma follows. (Here we have used $|x|$ for the absolute
value the complex number $x$.)
\end{proof}

Observe that if $H$ is a group of Lie type with field parameter $q$, then $q=q_0$ except for $^{2}B_2(q)$, $^{2}G_2(q)$ and $^{2}F_4(q)$ where $q=q_0^2$.

Finally, we conclude with a numerical lemma which we frequently use in what follows.

\begin{lemma}\label{pentagonal}Let $q\geq 2$. Then $\prod_{i= 1}^\infty\left(1-{q^{-i}}\right)>1-q^{-1}-q^{-2}$.
\end{lemma}
\begin{proof}
This is an immediate consequence of the Euler pentagonal number theorem, for a proof see~\cite[Lemma~$3.5$]{NP}.
\end{proof}

\section{Character degree sum for symmetric and alternating groups}

It is well-known that all irreducible representations of the symmetric group $\Sym(n)$ can be realized over the rational numbers~\cite[Theorem~$4.12$]{James}. Therefore,
by Lemma~\ref{is1}, the character degree sum $\Sigma(\Sym(n))$  equals the number of solutions of the equation $x^2=1$  in $\Sym(n)$.

We start by singling out the following lemma, which we will use quite often.

\begin{lemma}\label{Sylowp}Let $n\geq 1$, let $p$ be a prime and let $P$ be a Sylow $p$-subgroup of $\Sym(n)$. Then $|P|\leq 2^{n-1}$.
\end{lemma}

\begin{proof}
Let $n=a_0+a_1p+\cdots +a_kp^k$ be the $p$-adic expansion of $n$, that is, $a_0,\ldots,a_k\in \{0,\ldots,p-1\}$ with $a_k\neq 0$. Observe that a Sylow $p$-subgroup of $\Sym(p^i)$ has order $p^{\frac{p^i-1}{p-1}}$. Now, from the structure of $P$ and from an easy computation, it follows  that
\begin{equation}\label{eqsym1}|P|=\prod_{i=0}^kp^{a_i\frac{p^i-1}{p-1}}.
\end{equation}

We have
\begin{equation}\label{eqsym2}\sum_{i=0}^ka_i\frac{p^i-1}{p-1}=\frac{1}{p-1}\left(\sum_{i=0}^ka_ip^i-\sum_{i=0}^ka_i\right)=\frac{1}{p-1}\left(n-\sum_{i=0}^ka_i\right)\leq \frac{n-1}{p-1}.
\end{equation}
It is easy to verify that $x^{1/(x-1)}$ is a decreasing function of $x>0$, and hence $p^{1/(p-1)}\leq 2^{1/(2-1)}=2$. So, the proof follows from~\eqref{eqsym1} and~\eqref{eqsym2}.
\end{proof}

The upper bound in Lemma~\ref{Sylowp} is sharp for $p=2$ and for $n$ a power of $2$.

\begin{propo}\label{pp1} Let $n\geq 5$. Then, for every prime $p$, we have either $$\Sigma(\Sym(n))<|\Alt(n)|_{p'},$$
or one of the following holds
\begin{description}
\item[$(1)$]$n=5$ and $p\in \{2,5\}$;
\item[$(2)$]$n=6$ and $p\in \{2,3\}$;
\item[$(3)$]$n=8$ and $p=2$.
\end{description}
\end{propo}
\begin{proof}
Write $G=\Sym(n)$ and $S=\Alt(n)$.
In particular, from Lemma~\ref{Sylowp} we have
\begin{equation}\label{eq0}
|S|_{p'}\geq \frac{|S|}{|P|}=\frac{n!/2}{2^{n-1}}=\frac{n!}{2^n}.
\end{equation}

Write $\Sigma_n=\Sigma(\mathrm{Sym}(n))$. Lemma~$2$ in~\cite{Chowla} shows that
\begin{equation}\label{eq3}
\frac{\Sigma_{n}}{\Sigma_{n-1}}\leq n^{1/2}+1,
\end{equation}
for every $n\geq 1$ (where $\Alt(0)=\Sym(0)$ is the group of order $1$). Using~\eqref{eq3}, we get $\Sigma_n\leq 2\Sigma_{n-1}$ and hence we inductively obtain
\begin{equation}\label{eq5}
\Sigma_n\leq 2^n(n!)^{1/2}.
\end{equation}

Now it is a tedious computation to show that $$2^n(n!)^{1/2}<\frac{n!}{2^n},$$
for every $n\geq 40$. In particular, by~\eqref{eq0} and~\eqref{eq5}, the proof follows immediately for $n\geq 40$.
Another direct computation shows that, for every $n\in \{19,\ldots,40\}$, we have
$$\prod_{m=1}^n(m^{1/2}+1)<\frac{n!}{2^n}.$$
Hence, for $n\in \{19,\ldots,40\}$, the proof follows again from~\eqref{eq0} and~\eqref{eq5}.

Finally, the cases $5\leq n\leq 18$ can be checked one-by-one with the invaluable help of \texttt{magma}~\cite{magma} (observe that the character table of  $\Sym(n)$ is readily available in \texttt{magma} for every $n\leq 18$).
\end{proof}

\section{Exceptional groups of Lie type}\label{exceptional}

In this section $G$ is an exceptional group of Lie type having
adjoint isogeny type. This means that $G={\mathbf G}^{F}$, where
${\mathbf G}$ is an exceptional simple algebraic group of adjoint
isogeny type and $F$ is a Frobenius endomorphism of ${\mathbf
G}$. Note that $G$ is not necessarily simple.

The computer algebra project \texttt{CHEVIE}~\cite{chevie} for
symbolic computations with generic character tables of groups of
Lie type is a rather powerful tool for computing  the character
degree sum of $G$. In fact, we have the following result.

\begin{propo}\label{execp}Let $G$ be an
exceptional group of Lie type having adjoint isogeny. Then
$\Sigma(G)<C$, where $C$ is in the second row
of the column corresponding to $G$ in Table~$\ref{table1}$ (in the table, $q$ is the field parameter of $G$).
\end{propo}
\begin{proof}
The proof is an easy computer computation. Here we discuss in detail the case that $G={^{2}}F_4(q)$ (all other cases are similar). The character degrees together with their multiplicities in \texttt{magma}~\cite{magma} format are available in~\cite{Lubeck}. (These values were obtained with the computer algebra package \texttt{CHEVIE}.)  Now, a computation gives
\begin{eqnarray*}
\Sigma(G)&=&
q^{14}-q^{13}+\sqrt{2q}q^{12}+q^{12}-\sqrt{2q}q^{11}+q^{11}-\sqrt{2q}q^{10}+\frac{4}{3}q^{10}\\
&&+3\sqrt{2q}q^{9}-2q^{9}-\sqrt{2q}q^{8}+2q^{8}
-3\sqrt{2q}q^{7}-q^{7}+3\sqrt{2q}q^{6}-\frac{8}{3}q^{6}\\&&+\sqrt{2q}q^{5}+q^{5}
-3\sqrt{2q}q^4+\sqrt{2q}q^3-2q^3+\sqrt{2q}q^2+\frac{7}{3}q^2-\sqrt{2q}q.
\end{eqnarray*}
Again with the help of a computer (or with a direct computation)
we see that this number is at most $q^{14}$.
\end{proof}

\begin{propo}\label{exp1}Let $G$ be an exceptional group of Lie type of characteristic $\ell$. Then $\Sigma(G)<2\cdot|G|_{\ell'}$.
\end{propo}
\begin{proof}Using the information in Table~\ref{table1} and the order of $G$, with a case-by-case analysis and with  a computation we see that $C<2\cdot|G|_{\ell'}$ for $G\neq F_4(2)$. In particular, for $G\neq F_4(2)$, the proof follows from Proposition~\ref{execp}. 
Finally, using the character table of $F_4(2)$ in~\cite{Atl}, we also get $\Sigma(F_4(2))<2\cdot|F_4(2)|_{2'}$.
\end{proof}

\medskip
\begin{table}[!ht]
\begin{tabular}{|c|ccccc|cccccc|}\hline
$G$&$^{2}B_2(q)$ &$^{3}D_4(q)$ &$G_2(q)$ &$^{2}G_2(q)$&$F_4(q)$\\
 $C$&$q^3$            &$q^{16}+q^{13}$&
$q^8+\frac{3}{2}q^6$&$q^4$&$q^{28}+q^{27}$\\\hline\hline%% &&&&&\\ \hline
$G$&$^{2}F_4(q)$&$E_6(q)$&$^{2}E_6(q)$&$E_7(q)$&$E_8(q)$\\
$C$&$q^{14}$&$q^{42}+q^{38}$&$q^{42}+q^{39}$&$q^{70}+q^{67}$&$q^{128}+q^{125}$\\\hline
\end{tabular}\medskip
\caption{}\label{table1}
\end{table}

\begin{propo}\label{execbound}Let $G$ be a non-soluble exceptional group of Lie type
of characteristic $\ell$ and having adjoint isogeny type. Let $S$
be the non-abelian simple composition factor of $G$. If $p$ is a prime with $p\neq \ell$, then
either $\Sigma(G)<|S|_{p'}$, or
\begin{description}
\item[$(1)$]$G=G_2(2)$ and $p=3$, or
\item[$(2)$]$G={^2}G_2(3)$ and $p=2$.
\end{description}
\end{propo}

\begin{proof}
Suppose first that $G$ is not $^{2}B_2(q)$, $^{2}G_2(q)$ or $^{2}F_4(q)$, and let $r$ be the rank of the corresponding algebraic group. Let $W$ be the
Weyl group of ${\bf G}$ and let $P$ be a Sylow $p$-subgroup of $G$.
By Lemmas~\ref{ss1} and~\ref{tt2}, we have
$|P|\leq (q+1)^r|W|_p\leq (q+1)^r|W|$. Using the upper bound $C$
for $\Sigma(G)$ obtained in
Proposition~\ref{execp}, we see with a case-by-case analysis and
with the help of a computer that (for $G\neq G_2(2)$)
$$(\dag)\qquad C<\frac{|S|}{(q+1)^r|W|},$$ from which the lemma
immediately follows. Here we discuss in detail the case $G=G_2(q)$. Assume first that $q\neq 2$. So, $G=S$. In this case,
$C=q^8+3q^6/2$, $r=2$, $|W|=12$ and the inequality $(\dag)$
becomes
$$q^8+\frac{3}{2}q^6<\frac{q^6(q^6-1)(q^2-1)}{12(q+1)^2},$$ which
is easily seen to be true for $q\neq 2$. Assume now that $q=2$.
So, $S=(G_2(2))'\cong \mathrm{PSU}_3(3)$. Now,  we  use~\cite{Atl}
to see that the character degree sum of $G_2(2)$  is $328$ and
that $|S|_{p'}\leq 328$ only for $p=3$.

Suppose now that $G$ is one of $^{2}B_2(q)$ with $q=2^{2m+1}$ (for $m\geq 1$), $^{2}F_4(q)$ with $q=2^{2m+1}$ (for $m\geq 0$), or $^{2}G_2(q)$ with $q=3^{2m+1}$ (for $m\geq 0$). Write $q_0=\sqrt{q}$. By Lemmas~\ref{ss1} and~\ref{tt2}, we have
$|P|\leq (q_0+1)^r|W|_p\leq (q_0+1)^r|W|$. So, as in the previous paragraph, the lemma
follows with a case-by-case analysis using Proposition~\ref{execp} and the invaluable support of a computer. Here we give full details for $G=\,^{2}F_4(q)$. So, $S=G$ if $m>0$ and $|S|=|G|/2$ if $m=0$. By Proposition~\ref{execp}, we have $\Sigma(G)<q^{14}$. For $m>0$, it is a computation to see that the inequality
$$q^{14}<\frac{|S|}{1152(q_0+1)^4}=\frac{q^{12}(q^{6}+1)(q^4-1)(q^3+1)(q-1)(q^5+1)}{1152(q_0+1)^4}$$
is always satisfied. Finally, for $m=0$, we  use~\cite{Atl}.
\end{proof}

\section{Linear groups}\label{secPSL}
In this section we only deal with the projective general linear groups and we use some rather sharp results of Gow~\cite{G1}.

\begin{propo}\label{propoPSL}Let $n\geq 2$ be an integer, let $\ell$ be a prime  and let $q$ be a power of $\ell$. Let $G$ be the general linear group $\mathrm{GL}_n(q)$ and let $S$ be the projective special linear group $\PSL_n(q)$.  We have $\Sigma(G)< \left(1-{q^{-2}}-{q^{-4}}\right)^{-1}|G|_{\ell'}.$
Moreover, for $n\geq 4$, if $p$ is a prime with $p\neq \ell$, then
$\Sigma(G)<|S|_{p'}$.
\end{propo}
\begin{proof}
Gow in~\cite[Theorem~$4$]{G1} determines the exact value of $\Sigma(G)$ as a function of $q$ and $n$, when $\ell$ is odd. In fact, by factoring $q$ in the formulae for $\Sigma(G)$ in~\cite[Theorem~$4$]{G1}, we obtain
\begin{equation}\label{eqqodd}
\Sigma(G)=q^{\frac{n(n+1)}{2}}\prod_{\substack{i\,\textrm{odd}\\1\leq i\leq n}}\left(1-\frac{1}{q^i}\right).
\end{equation}
In a postscript~\cite[page~$505$]{G1}, Gow states that
~\cite[Theorem~$4$]{G1} has been proved independently (for even as
well as for odd $q$) by A.~A.~Klyachko \cite[Corollary 3.3]{K} and
by I.~G.~Macdonald. So,~\eqref{eqqodd} is valid also for $\ell=2$.

Now,
\begin{eqnarray}\label{PSL1}
|G|_{\ell'}&=&\prod_{i=1}^n(q^i-1)=q^{\frac{n(n+1)}{2}}\prod_{i=1}^n\left(1-\frac{1}{q^i}\right)\\\nonumber
&=&q^{\frac{n(n+1)}{2}}\prod_{
\substack{i\,\textrm{odd}\\1\leq i\leq n}
}\left(1-\frac{1}{q^i}\right)
\prod_{\substack{i\,\textrm{even}\\1\leq i\leq n}}\left(1-\frac{1}{q^i}\right)\\\nonumber
&>&\Sigma(G)\cdot \prod_{\substack{i\,\textrm{even}\\2\leq i<\infty}}\left(1-\frac{1}{q^i}\right)>\Sigma(G)(1-q^{-2}-q^{-4}),
\end{eqnarray}
where in the last inequality we have applied Lemma~\ref{pentagonal} with $q$ replaced by $q^2$.
 Now the first part of the lemma follows immediately from~\eqref{eqqodd} and~\eqref{PSL1}.

Let $p$ be a prime with $p\neq \ell$ and let $P$ be  a Sylow $p$-subgroup of $S$. Recall that by Lemma~\ref{Sylowp} we have $(n!)_{p}\leq 2^{n-1}$ and hence $|P|\leq (q+1)^{n-1}(n!)_{p}\leq(2(q+1))^{n-1}$ by Lemmas~\ref{ss1} and~\ref{tt2}. From this we get
\begin{eqnarray}\label{PSL3}
|S|_{p'}&\geq&\frac{|S|}{(2(q+1))^{n-1}}=\frac{q^{n^2-1}}{(n,q-1)(2(q+1))^{n-1}}\prod_{i=2}^{n}\left(1-\frac{1}{q^i}\right)\\\nonumber
&>&\frac{q^{n^2-1}}{(n,q-1)(2(q+1))^{n-1}}\prod_{i=2}^{\infty}\left(1-\frac{1}{q^i}\right)\\\nonumber
&\geq& \frac{q^{n^2-1}}{(n,q-1)(2(q+1))^{n-1}}(1-q^{-1}-q^{-2}),
\end{eqnarray}
where the last inequality follows again from Lemma~\ref{pentagonal}.
Observe that from~\eqref{eqqodd} we have $\Sigma(G)\leq q^{\frac{n(n+1)}{2}}(1-q^{-1})$. Now, for $n\geq 4$, it is a computation with a computer to see that the inequality
$$ \frac{q^{n^2-1}}{(n,q-1)(2(q+1))^{n-1}}(1-q^{-1}-q^{-2})>q^{\frac{n(n+1)}{2}}(1-q^{-1}),$$
is satisfied, except for $n=5$ and $q=2$, or $n=4$ and $q\leq 5$. In particular, apart this handful of exceptions, we see that the inequality $\Sigma(G)<|S|_{p'}$ follows from~\eqref{PSL3}.

For the remaining values of $n$ and $q$ the proof follows with a case-by-case inspection on each of the various possibilities.
 \end{proof}

It is easy to verify that $(1-q^{-2}-q^{-4})^{-1}$ is a decreasing function of $q\geq 2$ and hence it attains its maximum at $q=2$. In particular, it follows from Proposition~\ref{propoPSL} that the character degree sum of $G=\mathrm{GL}_n(q)$ is bounded above by $\frac{16}{11}|G|_{\ell'}$, where $\ell$ is the characteristic of $G$.

\begin{coro}\label{coroPSL}
Let $n\geq 2$ be an integer, let $\ell$ be a prime and let $q$ be a power of $\ell$. Let $G$ be the projective linear group $\mathrm{PGL}_n(q)$ and let $S$ be the projective special linear group $\mathrm{PSL}_n(q)$. If $p$ is a prime with $p\neq \ell$, then either $\Sigma(G)<|S|_{p'}$, or
\begin{description}
\item[$(1)$]$n=3$, $q=2$ and $p=7$,
\item[$(2)$]$n=2$, $\ell=2$ and $q+1=p^t$ for some $t\geq 1$, or
\item[$(3)$]$n=2$, $\ell$ is odd, $q+1=2\cdot p^t$ for some $t\geq 1$, or
\item[$(4)$]$n=2$, $\ell$ is odd, $p=2$ and $q-1=2^t$ for some $t\geq 1$.
\end{description}
\end{coro}
\begin{proof}
Let $p$ be a prime with $p\neq \ell$ and let $P$ be a Sylow $p$-subgroup of $S$.
Observe that
$$\Sigma(G)\leq\Sigma(\mathrm{GL}_n(q))$$
So, for $n\geq 4$, the proof follows from Proposition~\ref{propoPSL}.

Assume that $n=3$. Using the information in~\cite{Lubeck} (for instance), we see that $\Sigma(G)=q^2(q^3-1)$. Moreover, by Lemmas~\ref{ss1} and~\ref{tt2}, we have $|P|\leq (q+1)^2|W|_p\leq 3(q+1)^2$ and hence $|S|_{p'}\geq |S|/(3(q+1)^2)$. Now with a computation we see that the inequality $|S|/(3(q+1)^2)>\Sigma(G)$ is satisfied for every $q\geq 8$. The remaining cases can be easily handled with a case-by-case analysis: in each case we have that either $|S|_{p'}>\Sigma(G)$ or part~(1) holds.

Finally, assume that $n=2$.  Using the information in~\cite{Lubeck} (for instance), we see that $\Sigma(G)=q^2$ if $\ell=2$ and $\Sigma(G)=q^2+1$ if $\ell>2$. Suppose that $\ell=2$. Now $\Sigma(G)=q^2$ and $|S|=q(q^2-1)$. If $p$ divides $q-1$, then $|S|_{p'}\geq q(q+1)>q^2=\Sigma(G)$ and hence we may assume that $p$ divides $q+1$. If $q+1$ is not a power of $p$, then $|P|\leq (q+1)/3$ and $|S|_{p'}\geq 3q(q-1)>q^2=\Sigma(G)$. In particular, $|S|_{p'}\leq \Sigma(G)$ only when $q+1$ is a power of $p$ and part~(2) holds. Finally suppose that $\ell>2$ and recall that $\Sigma(G)=q^2+1$. It is an easy computation to see that if part~(3) or~(4) holds, then $|G|_{p'}\leq \Sigma(G)$. Now the rest of the lemma follows with an easy computation distinguishing the case whether $p=2$ or $p>2$.
\end{proof}

\section{Unitary groups}\label{secUnitary}
In this section we follow closely Section~\ref{secPSL} and we study the unitary groups.

\begin{propo}\label{propoPSU}
Let $n\geq 3$ be an integer, let $\ell$ be a prime and let $q$ be a power of $\ell$. Let $G$ be the general unitary group $\mathrm{GU}_n(q)$ and let $S$ be the projective special unitary group $\mathrm{PSU}_n(q)$. We have $\Sigma(G)<\left(1-{q^{-2}}-{q^{-4}}\right)^{-1}|G|_{\ell'}$.
Moreover, for $n\geq 4$, if $p$ is a prime with $p\neq \ell$, then either $\Sigma(G)<|S|_{p'}$, or
\begin{description}
\item[$(1)$]$n\in \{4,5,6\}$, $q=2$ and $p=3$, or
\item[$(2)$]$n=4$, $q=3$ and $p=2$.
\end{description}
\end{propo}
\begin{proof}
From~\cite[Theorem~$5.2$]{TV}, we have
\begin{eqnarray}\label{eqPSU0}
\Sigma(G)=\prod_{i=1}^n\left(q^i+\frac{1-(-1)^i}{2}\right)=\prod_{\substack{1\leq i\leq n\\i\,\mathrm{odd}}}(q^i+1)\prod_{\substack{1\leq i\leq n\\i\,\mathrm{even}}}q^i.
\end{eqnarray}
Moreover, using the order of $G$, we get
\begin{eqnarray*}
|G|_{\ell'}&=&\prod_{i=1}^n(q^{i}-(-1)^i)=\prod_{\substack{1\leq i\leq n\\i\,\mathrm{odd}}}(q^i+1)\prod_{\substack{1\leq i\leq n\\i\,\mathrm{even}}}(q^i-1)\\
&=&\prod_{\substack{1\leq i\leq n\\i\,\mathrm{odd}}}(q^i+1)\prod_{\substack{1\leq i\leq n\\i\,\mathrm{even}}}q^i\prod_{\substack{1\leq i\leq n\\i\,\mathrm{even}}}\left (1-\frac{1}{q^i}\right).
\end{eqnarray*}
So, from~\eqref{eqPSU0} and Lemma~\ref{pentagonal}, we get
$$\frac{|G|_{\ell'}}{\Sigma(G)}=\prod_{\substack{1\leq i\leq n\\i\,\mathrm{even}}}\left (1-\frac{1}{q^i}\right)>1-{q^{-2}}-{q^{-4}}.$$
Now, the first part of the proposition follows immediately.

For the second part of the proposition we first obtain an upper bound on $\Sigma(G)$ which will simplify some of our computations. From~\eqref{eqPSU0} we have
\begin{eqnarray}\label{eqPSU1}
\Sigma(G)&=&q^{\frac{n(n+1)}{2}}\prod_{\substack{1\leq i\leq n\\i\,\mathrm{odd}}}\left(1+\frac{1}{q^i}\right)<q^{\frac{n(n+1)}{2}}\prod_{i\,\mathrm{odd}}\left(1+\frac{1}{q^i}\right).
\end{eqnarray}

Define $z'=\prod_{i\geq 5,i\,\textrm{odd}}(1+q^{-i})$. Using the inequality $\log(1+x)\leq x$, we have
\begin{eqnarray}\label{PSUagain}
\log(z')&=&
\sum_{\substack{i\geq 5\\i\,\mathrm{odd}}}\log\left(1+\frac{1}{q^{i}}\right)\leq
\sum_{\substack{i\geq 5\\i\,\mathrm{odd}}}\frac{1}{q^{i}}=\frac{1}{q}\left(\sum_{j= 0}^\infty\frac{1}{q^{2j}}-1-\frac{1}{q^2}\right)\\\nonumber
&=&\frac{1}{q}\left(\frac{1}{1-1/q^2}-1-\frac{1}{q^2}\right)=\frac{1}{q^3(q^2-1)}.
\end{eqnarray}
Write $z=(1+q^{-1})(1+q^{-3})\exp((q^{3}(q^2-1))^{-1})$. From~\eqref{eqPSU1} and~\eqref{PSUagain}, it follows that
\begin{equation}\label{eqPSU2}
\Sigma(G)<q^{\frac{n(n+1)}{2}}z.
\end{equation}

Let $p$ be a prime with $p\neq \ell$ and let $P$ be a Sylow $p$-subgroup of $S$. Recall that from Lemma~\ref{Sylowp} we have $(n!)_{p}\leq 2^{n-1}$ and hence
$|P|\leq (q+1)^{n-1}(n!)_{p}\leq (q+1)^{n-1}2^{n-1}$ by Lemmas~\ref{ss1} and~\ref{tt2}. From this we get
\begin{eqnarray}\label{eqPSU3}
|S|_{p'}&=&\frac{|S|}{(2(q+1))^{n-1}}=\frac{q^{n^2-1}}{(n,q+1)(2(q+1))^{n-1}}\prod_{i=2}^n\left(1-\frac{1}{(-q)^i}\right)\\\nonumber
&>&\frac{q^{n^2-1}}{(n,q+1)(2(q+1))^{n-1}}\prod_{\substack{i\,\mathrm{even}}}\left(1-\frac{1}{q^i}\right)\\\nonumber
&>&\frac{q^{n^2-1}}{(n,q+1)(2(q+1))^{n-1}}(1-q^{-2}-q^{-4}),
\end{eqnarray}
where in the last inequality we used Lemma~\ref{pentagonal}.

Now an easy computation with the help of a computer shows that the inequality
$$q^{\frac{n(n+1)}{2}}z<\frac{q^{n^2-1}}{(n,q+1)(2(q+1))^{n-1}}(1-q^{-2}-q^{-4})$$
holds true, except for  $n\in \{5,6\}$ and $q=2$, or $n=4$ and $q\leq 7$. In particular, apart this handful of exceptions, we see from~\eqref{eqPSU2} and~\eqref{eqPSU3} that $\Sigma(G)<|S|_{p'}$.

For the remaining values of $n$ and $q$ the lemma follows with a direct inspection on each of the various possibilities.
\end{proof}

\begin{coro}\label{coroPSU}
Let $n\geq 3$ be an integer, let $\ell$ be a prime and let $q$ be a power of $\ell$. Let $G$ be the projective general unitary group $\mathrm{PGU}_n(q)$ and let $S$ be the projective special unitary group $\mathrm{PSU}_n(q)$. If $p$ is a prime with $p\neq \ell$, then either $\Sigma(G)<|S|_{p'}$, or
\begin{description}
\item[$(1)$]$n\in \{3,4\}$, $q=2$ and $p=3$, or
\item[$(2)$]$n=3$, $q=3$ and $p=2$.
\end{description}
\end{coro}
\begin{proof}
Observe that
$$\Sigma(G)<\Sigma(\mathrm{GU}_n(q)).$$
So, for $n\geq 4$, the proof follows from Proposition~\ref{propoPSU}. In fact, we may assume that either part~(1) or~(2) of Proposition~\ref{propoPSU} holds. Now the proof follows with a case-by-case direct inspection of $\mathrm{PGU}_4(2)$, $\mathrm{PGU}_5(2)$, $\mathrm{PGU}_6(2)$ and $\mathrm{PGU}_4(3)$ with the help of \texttt{magma}.

Assume that $n=3$. Using the information in~\cite{Lubeck} (for instance), we see that $\Sigma(G)=q^2(q^3+1)$.

 Let $p$ be a prime with $p\neq \ell$ and let $P$ be a Sylow $p$-subgroup of $S$. Now, $|S|_{\ell'}=(q-1)(q+1)^2(q^2-q+1)/(3,q+1)$. Moreover, by Lemmas~\ref{ss1} and~\ref{tt2}, we have $|P|\leq 3(q+1)^2$. Thus $|S|_{p'}\geq q^3(q-1)(q^2-q+1)/(3(3,q+1))$. Now, the inequality $$q^2(q^3+1)<\frac{q^3(q-1)(q^2-q+1)}{3(3,q+1)}$$
is satisfied for every $q\geq 9$. In particular, for $q\geq 9$, we have $\Sigma(G)<|S|_{p'}$. Finally, the remaining values of $q$ can be easily checked one-by-one.
\end{proof}

\section{Odd dimensional orthogonal groups}\label{secOrth}
In this section we follow the same pattern as in Sections~\ref{secPSL} and~\ref{secUnitary} and  we study the odd dimensional orthogonal groups.

\begin{propo}\label{propoPOmega}
Let $m\geq 3$ be an integer, let $\ell$ be an odd prime and let $q$ be a power of $\ell$. Let $G$ be the general orthogonal group $\mathrm{GO}_{2m+1}(q)$ and let $S$ be the simple orthogonal group $\Omega_{2m+1}(q)$. We have
$\Sigma(G)<\left((1-q^{-2})(1-{q^{-2}}-{q^{-4}})^2\right)^{-1}|G|_{\ell'}$.
Moreover, if $p$ is a prime with $p\neq \ell$, then $\Sigma(G)<|S|_{p'}$.
\end{propo}
\begin{proof}
Recall that $|G|=2q^{m^2}\prod_{i=1}^m(q^{2i}-1)$ and that $|S|=|G|/4$.

Given two non-negative integers $m$ and $k$ with $m\geq k$, the $q$-binomial coefficient ${m\choose k}_q$ is defined by
$${m\choose k}_q=\frac{(q^m-1)(q^{m-1}-1)\cdots (q^{m-k+1}-1)}{(q^k-1)(q^{k-1}-1)\cdots (q-1)}.$$
 From~\cite[Theorem~$4.1$]{V}, we see that
\begin{equation}\label{eqOodd}
\Sigma(G)=2\sum_{k=0}^mq^{2k(m-k+1)}{m\choose k}_{q^2}.
\end{equation}
Using this exact formula for $\Sigma(G)$ we now extract a rather sharp upper bound that will be useful for our proof. First,
\begin{eqnarray*}
{m\choose k}_{q^2}&\leq &\frac{q^{2m}\cdot q^{2(m-1)}\cdot\cdots \cdot q^{2(m-k+1)}}{(q^{2k}(1-q^{-2k}))\cdot(q^{2(k-1)}(1-q^{-2(k-1)}))\cdot\cdots\cdot (q^2(1-q^{-2}))}\\
&=&\frac{q^{2mk-(k-1)k}}{q^{k(k+1)}(1-q^{-2k})\cdots (1-q^{-2})}<\frac{q^{2mk-2k^2}}{\prod_{i=1}^{\infty}(1-q^{-2i})}\\
&<&q^{2k(m-k)}(1-q^{-2}-q^{-4})^{-1},
\end{eqnarray*}
where the last inequality follows from Lemma~\ref{pentagonal}. Write $z=(1-q^{-2}-q^{-4})^{-1}$. Combining this upper bound for ${m\choose k}_{q^2}$ with~\eqref{eqOodd}, we get
\begin{equation}\label{eqOodd1}
\Sigma(G)<2z\sum_{k=0}^mq^{4k(m-k)+2k}.
\end{equation}

Write $\{2t(m-t)+t\mid 0\leq t\leq m\}=\{x_0,x_1,\ldots,x_{m}\}$ with $x_0\leq x_1\leq \cdots \leq x_m$. It is easy to verify that the set $\{2t(m-t)+t\mid 0\leq t\leq m\}$ consists of exactly $m+1$ distinct non-negative integers, having minimum $0$ (achieved by taking $t=0$) and having maximum $m(m+1)/2$ (achieved by taking $t=\lfloor (m+1)/2\rfloor$). Therefore $x_0<x_1<\cdots<x_m=m(m+1)/2$ and hence $x_{m-t}\leq m(m+1)/2-t$, for each $t\in \{0,\ldots,m\}$. From this, it follows by~\eqref{eqOodd1} that
\begin{eqnarray}\label{eqOodd2}
\Sigma(G)&<&2z\sum_{t=0}^mq^{m(m+1)-2t}=2z{q^{m(m+1)}}\sum_{t=0}^mq^{-2t}\\\nonumber
&<&2zq^{m(m+1)}\sum_{t=0}^\infty q^{-2t}=2zq^{m(m+1)}\frac{q^2}{q^2-1}.
\end{eqnarray}
Moreover,
\begin{eqnarray*}
|G|_{\ell'}&=&2(q^2-1)\cdots (q^{2m}-1)=2q^{m(m+1)}(1-q^{-2})\cdots (1-q^{-2m})\\
&>&2q^{m(m+1)}\prod_{i=1}^{\infty}(1-q^{-2i})>2q^{m(m+1)}(1-q^{-2}-q^{-4})=2q^{m(m+1)}z^{-1}.
\end{eqnarray*}
Now, the first part of the proposition follows combining the above equation with~\eqref{eqOodd2}.

Let $p$ be a prime with $p\neq \ell$ and let $P$ be a Sylow $p$-subgroup of $S$. Now, the Weyl group $W$ of $G$ is the semidirect product of $\Sym(m)$ with the natural $\Sym(m)$-permutation module over the field of size $2$, that is, $|W|=2^{m}{m}!$. Hence, by Lemmas~\ref{Sylowp},~\ref{ss1} and~\ref{tt2}, we obtain
$|P|\leq (q+1)^m|W|_p\leq (q+1)^m2^{m}2^{m-1}=(q+1)^m2^{2m-1}$. From this and from Lemma~\ref{pentagonal} applied with $q$ replaced by $q^2$, it follows that
\begin{eqnarray}\label{eqOodd3}
|S|_{p'}&\geq &\frac{|G|/4}{(q+1)^m2^{2m-1}}=\frac{q^{m^2+m(m+1)}}{(q+1)^m2^{2m}}\prod_{k=1}^{m}(1-q^{-2k})\\\nonumber
&>&\frac{q^{m^2+m(m+1)}}{(4(q+1))^m}z^{-1}.
\end{eqnarray}

It is a computation with the usual help of a computer to see that the inequality
$$2zq^{m(m+1)}\frac{q^2}{q^2-1}<\frac{q^{m^2+m(m+1)}}{(4(q+1))^m}z^{-1}$$
is always satisfied. In particular, the proof follows from~\eqref{eqOodd2} and~\eqref{eqOodd3}.
\end{proof}

It is rather interesting to note that the error factor in front of $|G|_{\ell'}$ in Proposition~\ref{propoPOmega} is a decreasing function of $q$ which tends to $1$ as $q$ tends to infinity and having maximum (at $q=3$) roughly equal to $1.4642$.

\section{A uniform bound}\label{newbound}

In this section ${\mathbf G}$ is a simple classical group of rank $r$ and $W$ is its Weyl group.
Note that the Weyl group of the dual group ${\mathbf G}^*$ is isomorphic to $W$, and hence 
we use $W$ to denote the Weyl group of both ${\mathbf G}$ and ${\mathbf G}^*$. Observe also that ${\mathbf G}^*$ is a classical algebraic group.
Let $F$ be a Frobenius endomorphism of ${\mathbf G}$ and  $G={\mathbf G}^{F}$. 

\begin{lemma}\label{shalev}
Suppose that $G$ is a classical group. Let $v$ be the number of unipotent characters of $G$.
Then  $v\leq |W|$. Moreover, either $v\leq |W|^{1/\sqrt{r}}$, or  ${\mathbf G}$ is of type $B_r$ or $C_r$ with $2\leq r\leq 6$, or ${\mathbf G}$ is of type  $D_4$.
\end{lemma}
\begin{proof}
Given a non-negative integer $r$, we denote by $p(r)$ the number of partitions of $r$, and given a partition $\alpha=(\alpha_1,\ldots,\alpha_t)$, we write $|\alpha|=\sum_{i=1}^t\alpha_i$. From~\cite[page~$114$]{Pribitkin}, we see that, for $r\geq 1$, we have $$p(r)<\frac{e^{\pi\sqrt{2r/3}}}{r^{3/4}}.$$
We denote by $p^*(r)$ the function on the right-hand side of this inequality (where we also define $p^*(0)=1$). Now, it is an easy computation to see that $p^*(r)$ is an increasing function of $r$ and that the maximum of $\{p^*(a)p^*(r-a)\mid 0\leq a\leq r\}$ is achieved for $a=\lfloor r/2\rfloor $ with value bounded above by $(p^*(r/2))^2=e^{\pi\sqrt{4r/3}}/(r/2)^{3/2}$.

Stirling's formula~\cite{robbins} gives us that for every $r\geq 1$, $$r!\geq \sqrt{2\pi r}(r/e)^r.$$

Assume that $G$ is of Lie type $A_r$ (with $r \geq 1$) or $^{2}A_r$ (with $r\geq 2$). From~\cite[page~$465$]{Ca}, we see that the unipotent characters of $G$ are parametrized by  the partitions $\alpha$ of $r+1$. Hence
$v=p(r+1)\leq p^*(r+1)$.
Recall that $|W|=(r+1)!$. So, from Stirling's formula, we get $|W|=(r+1)r!\geq (r+1)\sqrt{2\pi r}(r/e)^r$ and hence $|W|^{1/\sqrt{r}}\geq ((r+1)\sqrt{2\pi r})^{1/\sqrt{r}}(r/e)^{\sqrt{r}}$. With the support of a computer we see that
$$p^*(r+1)\leq ((r+1)\sqrt{2\pi r})^{1/\sqrt{r}}(r/e)^{\sqrt{r}}$$
for every $r\geq 17$. Now, for $1\leq r\leq 16$, using the exact value of $|W|$ and of $v$ we see that $v\leq |W|^{1/\sqrt{r}}$.

Assume that $G$ is of Lie type $B_r$ (with $r\geq 2$) or $C_r$ (with $r \geq 3$). From~\cite[page~$467$]{Ca}, we see that the unipotent characters of $G$ are parametrized by ordered pairs of partitions $\alpha$ and $\beta$ with $|\alpha|+|\beta|=r-(s+s^2)$, where  $s$ runs through the set of non-negative integers with $s+s^2\leq r$. Hence
\begin{equation}\label{BlClexact}
v=\sum_{\substack{s\geq 0\\s+s^2\leq r}}\sum_{a=0}^{r-(s+s^2)}p(a)p(r-(s+s^2)-a).
\end{equation}
It follows that
\begin{eqnarray*}
v&\leq& \sum_{\substack{s\geq 0\\s+s^2\leq r}}\sum_{a=0}^{r-(s+s^2)}p^*(a)p^*(r-(s+s^2)-a)\leq\sum_{\substack{s\geq 0\\s+s^2\leq r}}\sum_{a=0}^{r-(s+s^2)}p^*(a)p^*(r-a)\\
&\leq& \sum_{\substack{s\geq 0\\s+s^2\leq r}}\sum_{a=0}^{r-(s+s^2)}\frac{e^{\pi\sqrt{4r/3}}}{(r/2)^{3/2}}\leq \sum_{\substack{s\geq 0\\s+s^2\leq r}}(r+1)\frac{e^{\pi\sqrt{4r/3}}}{(r/2)^{3/2}}\\
&\leq&(\sqrt{r}+1)(r+1)\frac{e^{\pi\sqrt{4r/3}}}{(r/2)^{3/2}}.
\end{eqnarray*}
Recall that $|W|=2^rr!$. So, from Stirling's formula, we get $|W|\geq 2^r \sqrt{2\pi r}(r/e)^r=\sqrt{2\pi r}(2r/e)^r$ and hence $|W|^{1/\sqrt{r}}\geq
(2\pi r)^{1/(2\sqrt{r})}(2r/e)^{\sqrt{r}}$. With the support of a computer we see that
$$(\sqrt{r}+1)(r+1)\frac{e^{\pi\sqrt{4r/3}}}{(r/2)^{3/2}}\leq (2\pi r)^{1/(2\sqrt{r})}(2r/e)^{\sqrt{r}},$$
for every $r\geq 44$. Now, for $7\leq r\leq 43$, using the exact value of $|W|$ and of $v$ (obtained with~\eqref{BlClexact}) we see that $v\leq |W|^{1/\sqrt{r}}$. Similarly, for $2\leq r\leq 6$, with another direct computation we see that $v\leq |W|$.

Assume that $G$ is of Lie type $^{2}D_r$ (with $r\geq 4$). From~\cite[page~$476$]{Ca}, we see that the unipotent characters of $G$ are parametrized by ordered pairs of partitions $\alpha$ and $\beta$ with $|\alpha|+|\beta|=r-s^2$, where $s$ runs through the set of odd positive integers with $s^2\leq r$. Hence
\begin{equation}\label{2Dlexact}
v=\sum_{\substack{s\,\textrm{odd}\\s^2\leq r}}\sum_{a=0}^{r-s^2}p(a)p(r-s^2-a).
\end{equation}
It follows that
\begin{eqnarray*}
v&\leq& \sum_{\substack{s\,\textrm{odd}\\s^2\leq r}}\sum_{a=0}^{r-s^2}p^*(a)p^*(r-s^2-a)\leq\sum_{\substack{s\,\textrm{odd}\\s^2\leq r}}\sum_{a=0}^{r-s^2}p^*(a)p^*(r-1-a)\\
&\leq& \sum_{\substack{s\,\textrm{odd}\\s^2\leq r}}\sum_{a=0}^{r-1}\frac{e^{\pi\sqrt{4(r-1)/3}}}{((r-1)/2)^{3/2}}\leq \sum_{\substack{s\,\textrm{odd}\\s^2\leq r}}r\cdot\frac{e^{\pi\sqrt{4(r-1)/3}}}{((r-1)/2)^{3/2}}\\
&\leq&\sqrt{r}\cdot r\cdot\frac{e^{\pi\sqrt{4(r-1)/3}}}{((r-1)/2)^{3/2}}=r^{3/2}\frac{e^{\pi\sqrt{4(r-1)/3}}}{((r-1)/2)^{3/2}}.
\end{eqnarray*}
Recall that $|W|=2^{r-1}r!$. So, from Stirling's formula, we get $|W|\geq 2^{r-1} \sqrt{2\pi r}(r/e)^r=\sqrt{\pi r/2}(2r/e)^r$ and hence $|W|^{1/\sqrt{r}}\geq (\pi r/2)^{1/(2\sqrt{r})}(2r/e)^{\sqrt{r}}$. With the support of a computer we see that
$$r^{3/2}\frac{e^{\pi\sqrt{4(r-1)/3}}}{((r-1)/2)^{3/2}}<
(\pi r/2)^{1/(2\sqrt{r})}(2 r/e)^{\sqrt{r}},$$
for every $r\geq 55$. Now, for $4\leq r\leq 54$, using the exact value of $|W|$ and of $v$ (obtained with~\eqref{2Dlexact}) we see that $v<|W|^{1/\sqrt{r}}$.

Finally, assume that $G$ is of Lie type $D_r$ (with $r\geq 4$). From~\cite[page~$472$]{Ca}, we see that the unipotent characters of $G$ of defect  $>0$ are parametrized by ordered pairs of partitions $\alpha$ and $\beta$ with $|\alpha|+|\beta|=r-s^2$, where $s$ runs through the set of even positive integers with $s^2\leq r$. Moreover, the unipotent characters of $G$ of defect $0$ are in one-to-one correspondence with the irreducible characters of $W$. The latter are described in~\cite[Proposition~$11.4.3$]{Ca}. It follows that
\begin{equation}\label{Dlexact}
v=\sum_{\substack{s\,\textrm{even}\\4\leq s^2\leq r}}\sum_{a=0}^{r-s^2}p(a)p(r-s^2-a)+\varepsilon,
\end{equation}
with
\[
\varepsilon=\left\{
\begin{array}{lcl}
\frac{1}{2}\sum_{a=0}^{r}p(a)p(r-a)&&\textrm{if }r \textrm{ odd},\\
\frac{1}{2}\sum_{a=0}^{r}p(a)p(r-a)+\frac{3}{2}p(r/2)&&\textrm{if }r \textrm{ even}.\\
\end{array}
\right.
\]
Now, arguing as above, it is a computer computation to show that $v\leq |W|^{1/\sqrt{r}}$ for $r\neq 4$. Finally, for $r=4$, we have $|W|=192>14=v$.
\end{proof}

Now suppose that ${\mathbf G}$ has connected centre. Denote by $G^*$ the dual group of $G$. Note that $\Irr (G)$
is the disjoint union of the so called Lusztig (geometric) series
${\mathcal E}_s$, where  $s$ runs through a set of representatives
for the conjugacy classes of semisimple elements  of $G^*$,
see~\cite[13.16]{DM}. The fact that $s$ runs over representatives
of the semisimple conjugacy classes in $G^*$ follows from the
observation that, for groups with connected centre,  the geometric
conjugacy class of $s$ in $G^*$ coincides with the ordinary conjugacy
class~\cite[page~136]{DM}.

\begin{theo}\label{e(s)} Let $G$ be a classical group of adjoint type,
let $s\in G^*$ be a semisimple element, and let $W$ be the Weyl
group of ${\mathbf G}$. Then $|{\mathcal E}_s|\leq |W|$.
\end{theo}

\begin{proof} Observe that $G^*$ is of simply connected type, and hence
$C_{{\mathbf G}^*}(s)$ is connected. Let ${\mathbf L}$ be the
semisimple component of $C_{{\mathbf G}^*}(s)$. As explained in~\cite[Lemma 3.4]{LS}, the number of unipotent characters of
$C_{ G^*}(s)$ equals that of ${\mathbf L}^{F}$, and coincides
with the number of characters in ${\mathcal E}_s$. 

Let ${\mathbf L}={\mathbf L}_1\circ \cdots \circ {\mathbf L}_k$ be the
decomposition of ${\mathbf L}$ as the central product of simple
components ${\mathbf L}_1,\ldots ,{\mathbf L}_k$.  It is known
that ${\bf L}_1,\ldots,{\bf L}_k$ are of classical type. Clearly, the sum of their ranks
 does not exceed the rank of ${\mathbf L}$, as well as the rank
of ${\mathbf G}$. Furthermore, the Frobenius endomorphism $F$ acts on the set $\{{\bf L}_1,\ldots,{\bf L}_k\}$ by permuting the
components. Let $m$ be the number of orbits of $F$ under this action. This implies that ${\mathbf L}^{F}$ is a central product of finite classical
groups $G_j$ ($j=1,\ldots ,m$), where    $G_j=({\mathbf
L}_{i_j})^{F^{a_{i_j}}}$ for some $i_j\in\{1,\ldots ,k\}$ and $a_{i_j}$ is the
size of the $F$-orbit on ${\mathbf L}_{i_j}$. 

By Lemma \ref{shalev},
the number of unipotent characters of $G_j$ does not exceed
$|W({\mathbf L}_{i_j})|$, where $W({\mathbf L}_{i_j})$ is the Weyl group
of ${\mathbf L}_{i_j}$. Therefore, the number of unipotent
characters of $C_{G^*}(s)$ does not exceed $\prod_{i=1}^k|W({\bf L}_i)|=|W({\mathbf L})|\leq |W|$. It follows that $|{\mathcal E}_s|\leq |W|$.
\end{proof}

\begin{propo}\label{np1}
Suppose $ G$ is a classical group
of adjoint isogeny type. Then $\Sigma(G)<
|G|_{\ell'} \cdot |W|$.
\end{propo}
\begin{proof}Let $G={\mathbf G}^{F}$, where   ${\mathbf G} $
is the algebraic group of adjoint isogeny type. Then the centre of
 ${\mathbf G}$ is trivial, and hence connected. Recall that $G^*$ is the dual group of $G$ and that $\Irr(G)=\cup_s{\mathcal E}_s$, where $s$ runs over representatives
of the semisimple conjugacy classes in $G^*$ (because the centre of ${\bf G}$ is connected). 

Every series ${\mathcal E}_s$ contains exactly one regular
character $\rho_s$ and exactly one semisimple character $\sigma_s$
\cite[14.47]{DM}. For each character $\chi\in {\mathcal E}_s$, the
degree $\chi(1)$ is of the form $\sigma_s(1)\cdot \nu_\chi(1)$,
where $\nu_\chi$ is a  unipotent character of $C_{G^*} (s) $
\cite[13.24]{DM}. It is shown in \cite[Theorem 1.2]{LMT} that
$\nu(1)\leq |C_{G^*}(s)|_{\ell}$, for every unipotent character
of $C_{{G^*}}(s)$. It also well-known that
$\rho_s(1)=\sigma_s(1)\cdot |C_{G^*}(s)|_{\ell}$. Thus
$\chi(1)\leq \rho_s(1)$, for every $\chi\in{\mathcal E}_s$.

Recall that the correspondence
$\chi\rightarrow \nu_\chi$ is bijective~\cite[13.23]{DM}. Therefore,
$|{\mathcal E}_s|$ equals the number 
 of unipotent characters of $C_{{G^*}}(s)$.
Therefore,
\begin{equation}\label{eqkey1}
\sum_{\chi\in\Irr (G)}\chi(1) \leq \sum_{s}|{\mathcal E}_s|\cdot
\rho_s(1),
\end{equation} where $s$  runs over a set of
representatives of the semisimple conjugacy classes in $G^*$. In fact, the inequality in~\eqref{eqkey1} is actually strict because for $s=1$ the Lusztig series $\mathcal{E}_1$ contains the trivial character of $G$ and the character of degree $|G|_\ell>1$.
By Proposition \ref{e(s)}, for every $s$ we have $|{\mathcal E}_s|\leq |W|$.
%where $e(1)$ is the number of the unipotent characters in $G^*$.
So, by Lemma~\ref{shalev} and ~\eqref{eqkey1}, we get
$$\Sigma(G)=\sum_{\chi\in \Irr(G)}\chi(1)<
\left(\sum_{s}\rho_s(1)\right)\cdot |W|=|G|_{\ell'}\cdot |W|,$$
%e^{\pi\sqrt{\frac{2n}{3}}}(n+1)(n+2)/2=
%|G|_{\ell'}e^{\pi\sqrt{\frac{2n}{3}}}(n+1)(n+2)/2,$$
 where the
last equality follows from~\cite[14.29 and 14.47~(ii)]{DM}.
\end{proof}

\section{The remaining classical groups and the
proof of Theorem~\ref{main1}}\label{andtherest}

In this section we consider the remaining classical groups, that
is, symplectic and even dimensional orthogonal groups. The main ingredient
in this section is Proposition~\ref{np1}. Potentially, we could have used Proposition~\ref{np1} for other (classical) groups. However, the upper bound offered in Proposition~\ref{np1} is rather crude when the rank is small, and hence it would leave us with a  long list of small cases. 

For symplectic and orthogonal groups with field parameter $q$ even we don't have an exact formula for
the character degree sum. So we content ourselves to just prove
the inequality relevant to our investigation. Nevertheless, we deal
uniformly with even and odd characteristic.

We first study the even dimensional orthogonal groups.
Proposition~\ref{np1}
 must be applied to ${\mathbf G}^{F}$, where ${\bf G}$ is of adjoint type. The centre of ${\bf G}^F$ is trivial, and ${\bf G}^F$ has a normal subgroup $S$, which is a
simple group. Specifically, $S\cong
\mathrm{P}\Omega_{2r}^\pm(q)$, where $r$ is the rank of ${\mathbf G}$. Moreover, when $q$
is even, $\Omega_{2r}^\pm(q)$ is centreless, and hence $S\cong
\Omega_{2r}^\pm(q)$.
 When $q$ is odd, $|{\bf G}^F|$
coincides with the order of $\mathrm{SO}^\pm_{2r}(q)$ (see~\cite{Ca}) and hence $|\mathrm{P}\Omega_{2r}^{\pm }(q)|=|\mathrm{SO}^\pm_{2r}(q)|/(4,q^r\mp 1)$  by~\cite[Tables 2.1.C and 5.1.A]{KL}.

\begin{propo}\label{ortheven}
Let $r\geq 4$ be an integer, let $\ell$ be a prime and let $q$ be
a power of $\ell$. Let $G^\pm$ be the orthogonal group of adjoint
isogeny type, %$\mathrm{GO}^{\pm}_{2m}(q)$
and let $S^\pm$ be the
simple orthogonal group $\mathrm{P}\Omega_{2r}^\pm(q)$. If $p$ is
a prime with $p\neq \ell$, then $\Sigma(G^\pm)<|S^\pm|_{p'}$.
\end{propo}
\begin{proof}
Recall that
$|G^\pm|=q^{r(r-1)}(q^r\mp 1)\prod_{i=1}^{r-1}(q^{2i}-1)$ and that the Weyl group $W$ of ${\bf G}$ has order
$2^{r-1}r!$. In particular, as $r!< r^{r-1}$, we get $|W|< (2r)^{r-1}$.
 It follows from Proposition~\ref{np1} that
\begin{equation}\label{ortheven1}
\Sigma(G)\leq|G^{\pm}|_{\ell'}\cdot |W|< |G^{\pm}|_{\ell'}(2r)^{r-1}.
\end{equation}

Let $p$ be a prime with $p\neq \ell$ and let $P$ be a Sylow $p$-subgroup of $S^\pm$. Now, as $W$ is isomorphic to a semidirect product of an elementary abelian $2$-group of order $2^{r-1}$ by the symmetric group $\Sym(r)$, we get from Lemmas~\ref{ss1},~\ref{tt2} and~\ref{Sylowp} that $|P|\leq (q+1)^r|W|_p\leq (q+1)^r2^{r-1}2^{r-1}=(4(q+1))^r/4$. In particular, from above, we have
\begin{equation}\label{easier}
|S^\pm|_{p'}\geq
\frac{4|S^\pm|}{(4(q+1))^r}
\geq
\frac{4(|G^\pm|/(4,q^r\mp 1
))}{(4(q+1))^r}\geq \frac{|G^\pm|}{(4(q+1))^r}=\frac{q^{r(r-1)}|G^\pm|_{\ell'}}{(4(q+1))^r}.
\end{equation}

Moreover, an easy computer computation shows that the inequality
$$\frac{q^{r(r-1)}}{(4(q+1))^r}>(2r)^{r-1}$$
is satisfied for $q=2$ and $r\geq 9$, for $q=3$ and $r\geq 6$, for $q\in\{4,5\}$ and $r\geq 5$, and for $q>7$. In particular, for these values of $q$ and $r$, the proof follows from~\eqref{ortheven1} and~\eqref{easier}.

For each of the remaining cases, by considering the prime factorization of $|S^\pm|$, we have checked that $|S^{\pm}|_{p'}>|W|\cdot|G^\pm|_{\ell'}$ for every prime $p\neq \ell$, except for $q=2$ and $r\leq 6$, and for $q=3$ and $r=4$. So, in view of Proposition~\ref{np1} we only need to consider these $8$ cases. For this we appeal to our last resource, as $r\leq 8$, the character degrees (together with their multiplicities) are available in~\texttt{magma}~\cite{magma} format  in~\cite{Lubeck}. Hence, another tedious computer computation concludes the proof.
\end{proof}

\begin{propo}\label{symp}
Let $r\geq 2$ be an integer, let $\ell$ be a prime and let $q$ be a power of $\ell$. Let $G$ be the symplectic group $\mathrm{Sp}_{2r}(q)$ and let $S$ be the group $\mathrm{PSp}_{2r}(q)$. If $p$ is a prime with $p\neq \ell$, then either $\Sigma(G)<|S|_{p'}$ or
\begin{description}
\item[$(1)$]$r=2$, $\ell=2$ and $p=3$,
\item[$(2)$]$r=2$, $\ell=3$ and $p=2$.
\end{description}
\end{propo}
\begin{proof}
The proof is exactly as the proof of Proposition~\ref{ortheven}. In particular, for small values of $r$, we heavily rely on the data in~\cite{Lubeck}.
\end{proof}

\begin{proof}[Proof of Theorem~$\ref{main1}$]
This follows immediately with a case-by-case analysis
using~\cite{Atl}, Lemma~\ref{is2},
Propositions~\ref{pp1},~\ref{execbound},~\ref{propoPOmega},~\ref{ortheven},~\ref{symp}
and Corollaries~\ref{coroPSL},~\ref{coroPSU}. Here we discuss in
detail the case $S=\mathrm{PSL}_2(q)$. If $\ell=2$, then $\mathrm{PSL}_2(q)=\mathrm{PGL}_2(q)$ and
part~(11) follows from Corollary~\ref{coroPSL}~(2). Assume that
$\ell>2$. Using the character table of $\PSL_2(q)$, we see that
$\Sigma(S)=(q^2+q+2)/2$ when $q\equiv 1\mod 4$, and $\Sigma(S)=q(q+1)/2$
when $q\equiv 3\mod 4$. Suppose that  Corollary~\ref{coroPSL}~(4)
is satisfied, that is, $p=2$ and $q-1=2^t$ for some $t\geq 2$. In
particular, $q\equiv 1\mod 4$. Now $|S|_{p'}=q(q+1)/2$ and hence
$\Sigma(S)=(q^2+q+2)/2>q(q+1)/2=|S|_{p'}$. So, part~(12) holds.
Finally, suppose that  Corollary~\ref{coroPSL}~(3) is satisfied,
that is, $q+1=2\cdot p^t$ for some $t\geq 1$. If $p>2$, then
$q\equiv 1\mod 4$ and
$|S|_{p'}=q(q-1)$, and hence $\Sigma(S)<|S|_{p'}$.  If $p=2$, then
$q\equiv 3\mod 4$ and
$|S|_{p'}=q(q-1)/2$, and hence $\Sigma(S)>|S|_{p'}$. So, part~(13)
holds.
\end{proof}

\end{document}